\documentclass[numbook,envcountsame,smallcondensed]{amsart}
\usepackage[foot]{amsaddr}
\usepackage{amssymb,amsmath,amsfonts,mathrsfs,cite,mathptmx,url}

\DeclareMathOperator{\alf}{alph}
\DeclareMathOperator{\simple}{sim}
\DeclareMathOperator{\mul}{mul}

\newtheorem{theorem}{Theorem}
\newtheorem{lemma}{Lemma}
\newtheorem{corollary}{Corollary}

\DeclareMathAlphabet{\mathcal}{OMS}{cmsy}{m}{n}

\makeatletter

\renewcommand*\subjclass[2][2010]{\def\@subjclass{#2}\@ifundefined{subjclassname@#1}{\ClassWarning{\@classname}{Unknown edition (#1) of Mathematics Subject Classification; using '2010'.}}{\@xp\let\@xp\subjclassname\csname subjclassname@#1\endcsname}}

\renewcommand{\subjclassname}{\textup{2020} Mathematics Subject Classification}

\makeatother

\begin{document}

\title{Small monoids generating varieties with uncountably many subvarieties}

\thanks{Supported by the Ministry of Science and Higher Education of the Russian Federation (project FEUZ-2023-0022).}

\author{Sergey V. Gusev}

\address{Ural Federal University, Institute of Natural Sciences and Mathematics, Lenina 51, Ekaterinburg 620000, Russia}

\email{sergey.gusb@gmail.com}

\begin{abstract}
An algebra that generates a variety with uncountably many subvarieties is said to be of \textit{type $2^{\aleph_0}$}.
We show that the Rees quotient monoid $M(aabb)$ of order ten is of type $2^{\aleph_0}$, thereby affirmatively answering a recent question of Glasson.
As a corollary, we exhibit a new example of type $2^{\aleph_0}$ monoid of order six, which turns out to be minimal and the first of its kind that is finitely based.
\end{abstract}



\maketitle

Recall that a \textit{variety} is a class of algebras of a fixed type that is closed under the formation of homomorphic images, subalgebras, and arbitrary direct products.
All varieties considered in the present article are varieties of monoids.
Following Jackson and Lee~\cite{Jackson-Lee-18}, a monoid is of \textit{type} $2^{\aleph_0}$ if the variety it generates contains uncountably many subvarieties.
In 2006, Jackson and McKenzie~\cite{Jackson-McKenzie-06} exhibited the first finite examples of type $2^{\aleph_0}$ monoids.
But since their examples are of order at least~20, it is of fundamental importance to find smaller examples.

For any set $\mathscr W$ of words over a countably infinite alphabet $\mathscr X$, let $M(\mathscr W)$ denote the Rees quotient of the free monoid $\mathscr X^\ast$ modulo the ideal of all words that are not subwords of any word in $\mathscr W$.
One of the central examples in Jackson and Lee~\cite{Jackson-Lee-18} is the monoid $M(abab)$ of order nine.
Not only is $M(abab)$ of type $2^{\aleph_0}$, it is the smallest example among all Rees quotients of $\mathscr X^\ast$.
Recently, Glasson~\cite{Glasson-24} proved that the monoid $M(abba)$ of order ten is also of type $2^{\aleph_0}$ and asked if the same holds true for the similar monoid $M(aabb)$, also of order ten.
One of the main goals of the present article is to answer this question in the affirmative.

\begin{theorem}
\label{T: M(aabb)}
The monoid $M(aabb)$ is of type $2^{\aleph_0}$.
\end{theorem}


Notice that, up to isomorphism, the monoids $M(aabb)$, $M(abab)$, and $M(abba)$ are the only type $2^{\aleph_0}$ monoids of the form $M(\mathbf w)$ of order ten or less.
Indeed, if there are at least two different letters, each of which occurs in $\mathbf w$ more than once, and $\mathbf w$ coincides (up to renaming of letters) with no of the words $aabb$, $abab$ or $abba$, then it is easy to see that $M(\mathbf w)$ is of order more than ten.
If at most one letter occurs in $\mathbf w$ more than once, then $M(\mathbf w)$ generates a variety with finitely many subvarieties \cite[Theorem~1.64]{Lee-23}.

Jackson and Lee~\cite{Jackson-Lee-18} also exhibited two type $2^{\aleph_0}$ monoids of order six:
\begin{align*}
\mathcal A_2^1: & =\langle a, b \mid aa = 0, \, aba = a, \, bab = bb = b\rangle\cup\{1\} \\
\text{and } \,\ \mathcal B_2^1: & =\langle a, b \mid aa = bb = 0, \, aba = a, \, bab = b\rangle\cup\{1\};
\end{align*}
see also Jackson and Zhang~\cite{Jackson-Zhang-21}.
These monoids are well-known to be inherently non-finitely based~\cite{Sapir-87}.
Since every monoid of order five or less generates a variety with at most countably many subvarieties \cite{Gusev-Li-Zhang-24,Lee-Zhang-14}, the monoids $\mathcal A_2^1$ and $\mathcal B_2^1$ are minimal examples of type $2^{\aleph_0}$ monoids.
Therefore, it is of natural interest to question the existence of other minimal examples of type $2^{\aleph_0}$ monoids; if such an example exists, then it is finitely based since up to isomorphism, $\mathcal A_2^1$ and $\mathcal B_2^1$ are the only monoids of order six or less that are non-finitely based~\cite{Lee-Li-11}.
We employ Theorem~\ref{T: M(aabb)} to answer this question in the affirmative.

\begin{corollary}
\label{C: A}
The monoid
\[ \mathcal M:=\langle a, e \mid ee = e, \, aaa = ae = 0, \, eaa = aa \rangle\cup\{1\} \]
of order six is of type $2^{\aleph_0}$.
\end{corollary}

\begin{proof} 
Lee and Li \cite[Proposition~8.1]{Lee-Li-11} have shown that the identities
\[
x^3\approx x^4, \, yzx^3\approx xyxzx, \, x^3y^3\approx y^3x^3, \, ytx^3y \approx ytyx^3, \, xyzxty\approx yxzxty,\, xzytxy\approx xzytyx
\]
constitute an identity basis for the monoid $\mathcal M$.
It is easy to check that the monoid $M(aabb)$ satisfies these identities and so belongs to the variety generated by $\mathcal M$.
The result then follows from Theorem~\ref{T: M(aabb)}.
\end{proof}

To prove Theorem~\ref{T: M(aabb)}, we need some definitions, notation and three auxiliary lemmas.
Recall that $\mathscr X^\ast$ is the free monoid over a countably infinite alphabet $\mathscr X$.
Elements of $\mathscr X$ are called \textit{letters} and elements of $\mathscr X^\ast$ are called \textit{words}.
The \textit{alphabet} of a word $\mathbf w$, i.e., the set of all letters occurring in $\mathbf w$, is denoted by $\alf(\mathbf w)$.
A letter $x\in\alf(\mathbf w)$ is \textit{simple} [respectively, \textit{multiple}] in a word $\mathbf w$ if $x$ occurs exactly once [respectively, at least twice] in $\mathbf w$; the set of all simple [respectively, multiple] letters of $\mathbf w$ is denoted by $\simple(\mathbf w)$ [respectively, $\mul(\mathbf w)$].
A non-empty word $\mathbf w$ is \textit{linear} if $\alf(\mathbf w)=\simple(\mathbf w)$.
For any $x\in\alf(\mathbf w)$, let $\mathbf w_x$ denote the word obtained from $\mathbf w$ by removing all occurrences of $x$ from $\mathbf w$.
The expression $_{i\mathbf w}x$ means the $i$th occurrence of a letter $x$ in a word $\mathbf w$.

We need the concept of the \textit{depth} of a letter $x$ in a word $\mathbf w$, introduced in~\cite[Chapter~3]{Gusev-Vernikov-18} to describe non-group varieties of monoids whose lattice of subvarieties forms a chain.
We define it by induction.
The letters of depth $0$ in $\mathbf w$ are precisely simple letters in $\mathbf w$.
Assume that the set of all letters of depth $k-1$ is defined.
We say that a letter $x$ in $\mathbf w$ is of depth $k$ if $x$ is not of depth less than $k$ and there is the first occurrence of some letter of depth $k-1$ between ${_{1\mathbf w}}x$ and ${_{2\mathbf w}}x$ in $\mathbf w$.
If, for any $n\ge0$, there are no first occurrences of letters of depth $n$ between ${_{1\mathbf w}}x$ and ${_{2\mathbf w}}x$ in $\mathbf w$, then we say that the depth of $x$ in $\mathbf w$ is $\infty$.
We will denote the depth of a letter $x$ in a word $\mathbf w$ by $D(\mathbf w,x)$.

\begin{lemma}[{\!\cite[Lemma~3.15]{Gusev-Vernikov-18}}]
\label{L: no div}
Let $\mathbf u=\mathbf a\varphi(\mathbf w)\mathbf b$, where $\mathbf a,\mathbf b\in\mathscr X^\ast$ and $\varphi\colon \mathscr X\to\mathscr X^\ast$ is a substitution.
If $x\in\alf(\mathbf w)$ and $D(\mathbf w, x)>0$, then $\varphi({_{1\mathbf w}}x)$ considered as a subword of $\mathbf u$ does not contain any first occurrence of a letter of depth less than $D(\mathbf w, x)$ in $\mathbf u$.\qed
\end{lemma}

As usual, let $\mathbb N$ denote the set of all positive integers.
For each $n\in\mathbb N$, let
\[
\mathbf w_n:=\biggl(\prod_{i=1}^n z_it_i\biggr)\,x\,\biggl(\prod_{i=1}^n z_iy_i^{(n)}\biggr)\,x\,\biggl(\prod_{j=1}^n\biggl(\prod_{i=1}^n y_i^{(n-j)}y_i^{(n+1-j)}\biggr)\biggr).
\]
Notice that the words $\mathbf w_n$ is a combination of two word patterns introduced by Jackson~\cite{Jackson-05} and Lee and Zhang~\cite{Lee-Zhang-14}, respectively.
These two word patterns and their variants have been employed in many further articles; see~\cite{Gusev-Vernikov-18,Gusev-Li-Zhang-24,Jackson-Lee-18}, for instance.

\begin{lemma}
\label{L: D(w_n,..)}
For each $n\in\mathbb N$, $1\le i\le n$ and $0\le k\le n$, we have $D(\mathbf w_n,x)=n+1$, $D(\mathbf w_n, y_i^{(k)})=k$, $D(\mathbf w_n,t_i)=0$ and $D(\mathbf w_n,z_i)=1$.
\end{lemma}

\begin{proof}
Evidently, $\simple(\mathbf w_n)=\{t_1,\dots,t_n,y_1^{(0)},\dots,y_n^{(0)}\}$.
Hence $D(\mathbf w_n,t_i)=D(\mathbf w_n,y_i^{(0)})=0$ for any $i=1,\dots,n$.
Since $z_i\in\mul(\mathbf w_n)$ and the simple letter $t_i$ lies between ${_{1\mathbf w_n}}z_i$ and ${_{2\mathbf w_n}}z_i$ in $\mathbf w_n$, we have $D(\mathbf w_n,z_i)=1$.
Further, assume that $D(\mathbf w_n,y_i^{(j)})=j$ for any $0\le j<k$.
Since the letter $y_i^{(k)}$ is not of depth less than $k$ and the letter ${_{1\mathbf w_n}}y_i^{(k-1)}$ of depth $k-1$ is between ${_{1\mathbf w_n}}y_i^{(k)}$ and ${_{2\mathbf w_n}}y_i^{(k)}$ in $\mathbf w_n$, we have $D(\mathbf w_n,y_i^{(k)})=k$.
Finally, the letter $x$ is not of depth less than $n$ and there is ${_{1\mathbf w_n}}y_1^{(n)}$ between ${_{1\mathbf w_n}}x$ and ${_{2\mathbf w_n}}x$ in $\mathbf w_n$.
Hence $D(\mathbf w_n,x)=n+1$.
\end{proof}

For any $N\subseteq\mathbb N$, put $\mathscr W_N:=\{\mathbf w_n\mid n\in N\}$.

\begin{lemma}
\label{L: continuum}
Let $N$ be a subset of $\mathbb N$.
If $n\notin N$, then $M(\mathscr W_N)$ satisfies $\mathbf w_n\approx x^2(\mathbf w_n)_x$.
\end{lemma}

\begin{proof}
Lemma~5.1 in~\cite{Jackson-Sapir-00} reduces our considerations to the case when the set $N$ is singleton, i.e., $N=\{k\}$ for some $k\ne n$.
Consider an arbitrary substitution $\varphi\colon \mathscr X\to M(\mathbf w_k)$.
We are going to show that $\varphi(\mathbf w_n)=\varphi(x^2(\mathbf w_n)_x)$.
If $\varphi(x)=1$, then the required claim is evident.
Now let $\varphi(x)\ne 1$.
Then $\varphi(x^2(\mathbf w_n)_x)=0$ because the word $\mathbf w_k$ is square-free.
Arguing by contradiction, suppose that $\varphi(\mathbf w_n)$ is a subword of $\mathbf w_k$, i.e., there are $\mathbf a,\mathbf b\in\mathscr X^\ast$ such that $\mathbf w_k=\mathbf a\varphi(\mathbf w_n)\mathbf b$.

Notice that every subword of length $2$ of $\mathbf w_k$ has the unique occurrence in $\mathbf w_k$.
Since each letter occurs in $\mathbf w_k$ at most twice, it follows that
\begin{itemize}
\item[\textup{($\ast$)}] for any $c\in\mul(\mathbf w_n)$, either $\varphi(c)=1$ or $(\varphi({_{1\mathbf w_n}}c),\varphi({_{2\mathbf w_n}}c))=({_{1\mathbf w_k}}d,{_{2\mathbf w_k}}d)$ for some $d\in\mul(\mathbf w_k)$.
\end{itemize}
In particular, $\varphi({_{1\mathbf w_n}}x)$ is the first occurrence of some letter in $\mathbf w_k$.
Since, by Lemma~\ref{L: D(w_n,..)}, $D(\mathbf w_n,x)=n+1$ and the depth of each letter in $\mathbf w_k$ is at most $k+1$, this fact and Lemma~\ref{L: no div} imply that $n\le k$.

Further, it is evident that the word $\varphi(xz_1y_1^{(n)}\cdots z_ny_n^{(n)}x)$ is not linear and contains no simple letters of $\mathbf w_k$.
By~($\ast$), this word is of length at most $2n+2$.
However, every subword of $\mathbf w_k$ of length less than $2k+2$ not containing simple letters of $\mathbf w_k$ is linear.
Hence $k\le n$, which contradicts the assumption that $n\ne k$.
\end{proof}

\begin{proof}[Proof of Theorem~\ref{T: M(aabb)}]
O.\@ Sapir \cite[Lemma~3.2(i)]{Sapir-19} has shown that the set
\[
\Sigma:=
\{x^3\approx x^4,\,x^3y\approx yx^3,\, yzx^3\approx xyxzx,\, xyzxty\approx yxzxty,\, xzytxy\approx xzytyx\}
\]
forms an identity basis for the monoid $M(aabb)$.
We are going to verify that $M(\mathscr W_{\mathbb N})$ satisfies each identity in $\Sigma$.
Indeed, consider an arbitrary substitution $\varphi\colon \mathscr X\to M(\mathscr W_{\mathbb N})$ and an arbitrary identity $\mathbf u\approx \mathbf v$ in $\Sigma$.
Since $\mathbf u_x=\mathbf v_x$ in any case, we have $\varphi(\mathbf u)=\varphi(\mathbf v)$ whenever $\varphi(x)=1$.
So, we may further assume that $\varphi(x)\ne1$.
Then if $\mathbf u\approx \mathbf v$ is one of the first three identities in $\Sigma$, then $\varphi(\mathbf u)=\varphi(\mathbf v)=0$ because each word $\mathscr W_{\mathbb N}$ contains at most two occurrences of any letter.
Assume now that $\mathbf u\approx \mathbf v$ is one of the identities $xyzxty\approx yxzxty$ or $xzytxy\approx xzytyx$.
If $\varphi(y)=1$, then $\varphi(\mathbf u)=\varphi(\mathbf v)$ because $\mathbf u_y=\mathbf v_y$ in any case.
Now let $\varphi(y)\ne1$.
Then neither $\varphi(xy)$ nor $\varphi(yx)$ is a subword of a word in $\mathscr W_{\mathbb N}$ because every subword of length 2 of $\mathbf w_n$  consists of the first occurrence of a letter and the last occurrence of a letter in $\mathbf w_n$, whence $\varphi(\mathbf u)=\varphi(\mathbf v)=0$ again.
Thus, $M(\mathscr W_{\mathbb N})$ satisfies $\Sigma$.
 
It remains to verify that the monoid $M(\mathscr W_{\mathbb N})$ is of type $2^{\aleph_0}$.
If $N_1$ and $N_2$ are distinct subsets of $\mathbb N$, then we may assume without any loss that there is $n\in N_1\setminus N_2$.
According to Lemma~\ref{L: continuum}, the monoid $M(\mathscr W_{N_2})$ satisfies the identity $\mathbf w_n\approx x^2(\mathbf w_n)_x$.
However this identity does not hold in $M(\mathscr W_{N_1})$.
Therefore, the monoids $M(\mathscr W_{N_1})$ and $M(\mathscr W_{N_2})$ generate distinct varieties.
Since the set of all subsets of a countably infinite set is uncountable and any monoid of the form $M(\mathscr W_N)$ is a quotient of $M(\mathscr W_{\mathbb N})$, the monoid $M(\mathscr W_{\mathbb N})$, and therefore also the monoid $M(aabb)$, are of type $2^{\aleph_0}$.
The proof of Theorem~\ref{T: M(aabb)} is thus complete.
\end{proof}



\begin{thebibliography}{99}
\bibitem{Glasson-24}
Glasson, D.: The Rees quotient monoid $M(abba)$ generates a variety with uncountably many subvarieties. Semigroup Forum \textbf{109}, 476--481 (2024). \url{https://doi.org/10.1007/s00233-024-10463-5}

\bibitem{Gusev-Li-Zhang-24}
Gusev, S.V., Li, Y.X., Zhang, W.T.: Limit varieties of monoids satisfying a certain identity. Algebra Colloq., to appear. Preprint \url{https://arxiv.org/abs/2107.07120v2}

\bibitem{Gusev-Vernikov-18}
Gusev, S.V., Vernikov, B.M.: Chain varieties of monoids. Dissert. Math. \textbf{534}, 1--73  (2018). \url{https://doi.org/10.4064/dm772-2-2018}

\bibitem{Jackson-05}
Jackson, M.: Finiteness properties of varieties and the restriction to finite algebras. Semigroup Forum \textbf{70}, 159--187 (2005). \url{https://doi.org/10.1007/s00233-004-0161-x}

\bibitem{Jackson-Lee-18}
Jackson, M., Lee, E.W.H.: Monoid varieties with extreme properties. Trans. Amer. Math. Soc. \textbf{370}, 4785--4812 (2018). \url{https://doi.org/10.1090/tran/7091}

\bibitem{Jackson-McKenzie-06}
Jackson, M., McKenzie, R.: Interpreting graph colorability in finite semigroups. Int. J. Algebra Comput.  \textbf{16}, 119--140 (2006). \url{https://doi.org/10.1142/S0218196706002846}

\bibitem{Jackson-Sapir-00}
Jackson, M., Sapir, O.: Finitely based, finite sets of words. Int. J. Algebra Comput. \textbf{10}, 683--708 (2000). \url{https://doi.org/10.1142/S0218196700000327}

\bibitem{Jackson-Zhang-21}
Jackson, M., Zhang, W.T.: From $A$ to $B$ to $Z$. Semigroup Forum \textbf{103}, 165--190 (2021). \url{https://doi.org/10.1007/s00233-021-10180-3}

\bibitem{Lee-23}
Lee, E.W.H.: Advances in the Theory of Varieties of Semigroups. Birkh\"{a}user/Springer, Cham (2023). \url{https://doi.org/10.1007/978-3-031-16497-2}

\bibitem{Lee-Li-11}
Lee, E.W.H., Li, J.R.: Minimal non-finitely based monoids. Dissert. Math. \textbf{475}, 1--65 (2011). \url{https://doi.org/10.4064/dm475-0-1}

\bibitem{Lee-Zhang-14}
Lee, E.W.H., Zhang, W.T.: The smallest monoid that generates a non-Cross variety. Xiamen Daxue Xuebao Ziran Kexue Ban \textbf{53}, 1--4 (2014) [In Chinese]. \url{https://doi.org/10.6043/j.issn.0438-0479.2014.01.001}

\bibitem{Sapir-87}
Sapir, M.V.: Problems of Burnside type and the finite basis property in varieties of semigroups. Izv. Akad. Nauk SSSR Ser. Mat.~\textbf{51}, 319--340 (1987)  [In Russian; English translation: Math. USSR, Izv.~\textbf{30}, 295--314 (1988)].  \url{https://doi.org/10.1070/IM1988v030n02ABEH001012}

\bibitem{Sapir-19}
Sapir, O.: Finitely based sets of 2-limited block-2-simple words. Semigroup Forum~\textbf{99}, 881--897 (2019). \url{https://doi.org/10.1007/s00233-019-10063-8}
\end{thebibliography}
\end{document}